\documentclass{amsart}
\usepackage{amssymb}

\usepackage[british,UKenglish,USenglish,american]{babel}
\usepackage{graphicx}           
\usepackage{fancyhdr}
\usepackage{rotating}
\usepackage{amsmath}
\usepackage{amssymb}
\usepackage{amsthm}
\usepackage{tikz}
\usepackage{url}
\usepackage{enumerate}
\usepackage{mathtools}
\usepackage{setspace}
\usepackage{color}
\usepackage[normalem]{ulem}
\usetikzlibrary{positioning}

\newtheorem{theorem}{Theorem}[section]

\newtheorem{fact}[theorem]{Fact}
\newtheorem*{theorem*}{Theorem}
\newtheorem*{maintheorem*}{Main Theorem}
\newtheorem*{lemma*}{Lemma}

\theoremstyle{definition}

\newtheorem{question}[theorem]{Question}
\newtheorem{definition}[theorem]{Definition}
\newtheorem{remark}[theorem]{Remark}
\newtheorem*{remark*}{Remark}

\newtheorem{claim}{Claim}
\newtheorem*{claim*}{Claim}

\title[Pseudofinite groups of finite centraliser dimension]{A note on pseudofinite groups of finite centraliser dimension}

\author[U. Karhum\"{a}ki]{Ulla Karhum\"{a}ki}
\address{University of Helsinki}
\email{ulla.karhumaki@helsinki.fi}

\thanks{The author is funded by the Finnish Science Academy grant no: 322795.}

\begin{document}
\maketitle

\begin{abstract} We give a structural theorem for pseudofinite groups of finite centraliser dimension. As a corollary, we observe that there is no finitely generated pseudofinite group  of finite centraliser dimension. 
\end{abstract}

\section{Introduction}

An infinite group is called \emph{pseudofinite} if every first-order sentence true in it also holds in some finite group or, equivalently, if it is elementarily equivalent (denoted by $\equiv$) to a non-principal ultraproduct of finite groups. 

Let $G$ be a group and $k$ be a positive integer. Then $G$ is said to be of \emph{centraliser dimension} $k$ if the longest proper descending chain of centralisers in $G$ has length $k$---this is denoted by ${\rm cd}(G)=k$. If such an integer $k$ exists, then $G$ is called a group of \emph{finite centraliser dimension}. 

By results of Wilson \cite{Wilson2009}, the solvable radical is uniformly definable in the class of \emph{finite} groups. Let ${\rm Rad}(H)$ be the solvable radical of a finite group $H$. Then, for a pseudofinite group $G \equiv \prod_{i\in I}G_i/U$, we have the well-defined definable normal subgroup $R(G)\equiv \prod_{i\in I} {\rm Rad}(G_i)/U$. By results of Khukhro \cite{Khukhro2009}, if $G$ is of centraliser dimension $k$ then the subgroup $R(G)$ is solvable of $k$-bounded derived length.

The \emph{socle} of a group $G$, denoted by ${\rm Soc}(G)$, is the subgroup generated by all minimal normal non-trivial subgroups of $G$.

We prove the following theorem.

\begin{theorem}\label{th:pseudofinite_fcd}Let $G \equiv \prod_{i\in I} G_i/U$ be a pseudofinite group with ${\rm cd}(G)=k$. Then $G$ has normal series of definable subgroups $$1 \unlhd R(G) \unlhd L \unlhd G,$$ where\begin{enumerate}
\item \label{radical} \emph{(Khukhro \cite{Khukhro2009}).} $R(G)$ is a solvable group of $k$-bounded derived length.
\item \label{pre-imge}$L$ is the full pre-image of $\overline{L} \equiv \prod_{i\in I}{\rm Soc}(G_i/{\rm Rad}(G_i))/U$ in $G$.
\item \label{socle}$\overline{L}=\overline{L}_1 \times \cdots \times \overline{L}_\ell$ is a direct product of $\ell$ many simple non-abelian groups $\overline{L}_j$, each of which is either finite or isomorphic to a (twisted) Chevalley group over a pseudofinite field. Also, $\ell$ depends only on $k$.
\item \label{centraliser}$C_{G/R(G)}(\overline{L})=1$.
\end{enumerate} 
\end{theorem}

In \cite{Tent-Macpherson2007}, Macpherson and Tent proved that if $G$ is a \emph{stable }pseudofinite group then the quotient $G/R(G)$ is finite; in this case, $R(G)$ is solvable as stable groups are of finite centraliser dimension. In \cite{Tent-Macpherson2012} they proved more generally that if $G$ is a pseudofinite group whose theory is $NIP$ then $G/R(G)$ is finite; here $R(G)$ needs not to be solvable. In \cite{Macpherson2018}, Macpherson proved a similar theorem to Theorem~\ref{th:pseudofinite_fcd}$(\ref{pre-imge})$-$(\ref{socle})$ for pseudofinite groups whose theory is $NTP_2$. These model-theoretic tameness assumptions (stable, ${\rm NIP}$, ${\rm NTP_2}$) on $G$ are inherited by definable quotients, such as $G/R(G)$. 

The property of being of finite centraliser dimension is \emph{not} inherited by (definable) quotients in general. However, if $H$ is a \emph{locally finite} group of centraliser dimension $k$ then the quotient $H/{\rm Rad}(H)$ is of $k$-bounded centraliser dimension. This fact is proven by Borovik and Karhum\"{a}ki \cite{Borovik-Karhumaki2019} and by Buturlakin \cite{Buturlakin2019} in their independent proofs describing the structure of locally finite groups of finite centraliser dimension. It immediately follows that if $G$ is a pseudofinite group of centraliser dimension $k$ then the quotient $G/R(G)$ is of $k$-bounded centraliser dimension.

In the proof of Theorem~\ref{th:pseudofinite_fcd} we use (the rather standard) techniques present (at least) in \cite{Tent-Macpherson2007, Tent-Macpherson2012, Macpherson2018, Borovik-Karhumaki2019}. Namely, we observe that if the definable quotient $G/R(G)$ (which is of finite centraliser dimension by \cite{Borovik-Karhumaki2019, Buturlakin2019}) is infinite, then it contain a definable normal subgroup isomorphic to a simple pseudofinite group.

It is an open question due to Sabbagh whether finitely generated pseudofinite groups exist. Neither solvable nor simple finitely generated pseudofinite groups exist (see e.g. \cite[Proposition 3.14]{Point-Houcine2013}). By invoking Theorem~\ref{th:pseudofinite_fcd}, we observe that no finitely generated pseudofinite group of finite centraliser dimension exists.

\begin{theorem}\label{th:finite-gene-fcd} There is no finitely generated pseudofinite group of finite centraliser dimension.
\end{theorem}

This note is organised as follows. In Section~\ref{Sec:pseudofinite groups} we give all necessary background results on (pseudo)finite groups. Then, in Section~\ref{Sec:Proofs}, we prove Theorem~\ref{th:pseudofinite_fcd} and Theorem~\ref{th:finite-gene-fcd}.

\section{Background results}\label{Sec:pseudofinite groups}

\subsection{Pseudofinite groups and ultraproducts} Pseudofinite groups arise in model theory\footnote{We omit the basics of model theory; for such detail we refer the reader to \cite{Ziegler-Tent2012}.} as groups satisfying the first-order properties that are true in all finite groups. Formally, the definition goes as follows.

\begin{definition}An infinite group $G$ (resp. structure $M$) is called \emph{pseudofinite} if every first-order sentence true in $G$ (resp. $M$) also holds in some finite group (resp. structure).\end{definition}

The group $(\mathbb{Z}, +)$ is an example of an infinite group which is \emph{not} pseudofinite; a first-order sentence expressing `if the map $x \mapsto 2x$ is injective then it is surjective' is true in all finite groups but false in $(\mathbb{Z}, +)$. Typical examples of pseudofinite groups are torsion-free divisible abelian groups, infinite extraspecial groups of exponent $p >2$ and rank $n$ and (twisted) Chevalley groups over pseudofinite fields.

\begin{remark}---\begin{itemize}
\item We always consider groups (resp. fields) in the pure group (resp. field) language. Also, definable means definable possibly with parameters.
\item Some authors allow pseudofinite groups to be finite. We require that a pseudofinite group is infinite. Therefore, given a pseudofinite group $G$, any definable subgroup of $G$ is either finite or pseudofinite. Moreover, given a definable normal subgroup $N$ of $G$, each of the groups $N$ and $G/N$ is either finite or pseudofinite (see \cite[Lemma 2.16]{Point-Houcine2013}).

\end{itemize} 
\end{remark}

Let $\{ M_i : i \in I \}$ be a family of $L$-structures, $U$ be a non-principal ultrafilter on the index set $I$ and  $M^*:= \prod_{i \in I} M_i$ be the Cartesian product of the $M_i$'s. We say that a property $\star$ holds \emph{for almost all $i$} if $\{i : \star \, \, {\rm holds \, \, for}  \, \, M_i \} \in U$. We fix $M=M^* / \sim_{U}$ for an equivalence relation $\sim_{U}$ on $M^*$ defined as: \[x \sim_{U} y \, \, \, {\rm if \, \, and \, \, only \, \, if} \, \, \, \{i \in I : x(i)=y(i)\} \in U,\]
where $x,y \in M^*$, and $x(i)$ and $y(i)$ denote the $i^{th}$ coordinate of $x$ and $y$, respectively. This is well-defined, and the resulting $L$-structure $M= \prod_{i \in I} M_i / U$ is called the \emph{ultraproduct} of the $L$-structures $M_i$ over the ultrafilter $U$. 

The famous \L o\'{s} theorem (see e.g. \cite[Exercise 1.2.4]{Ziegler-Tent2012}) states that a first-order expressible property is true in $M=\prod_{i \in I} M_i / U$ if and only if it is true in $M_i$, for almost all $i$. Using this, one easily observes that a structure is pseudofinite if and only if it is elementarily equivalent to a non-principal ultraproduct of finite structures of increasing sizes.

\subsubsection{Simple pseudofinite groups}In \cite{Ax1968}, Ax proved a beautiful purely algebraic characterisation for pseudofinite fields. While no such characterisation exists for pseudofinite groups, there is a tight correspondence between pseudofinite fields and simple pseudofinite groups; in \cite{Wilson1995}, Wilson proved that a simple pseudofinite group is elementarily equivalent to a (twisted) Chevalley group over a pseudofinite field. In \cite[Chapter 5]{Ryten2007}, Ryten proved that `elementarily equivalent' can be replaced by `isomorphic' in Wilson's result. As a (twisted) Chevalley group over a pseudofinite field is a simple pseudofinite group \cite{Point1999}, we have the following classification.

\begin{theorem}[Wilson \cite{Wilson1995} + Ryten {\cite{Ryten2007}}]\label{fact:Wilson-and-Ryten} A pseudofinite group is simple if and only if it is isomorphic to a (twisted) Chevalley group over a pseudofinite field.
\end{theorem}

A group is called \emph{definably simple} if it contains no proper non-trivial \emph{definable} normal subgroups. In \cite{Ugurlu2013}, U\u{g}urlu proved that a definably simple non-abelian pseudofinite group of finite centraliser dimension is simple. In her proof, she in particular uses the following important result, which will be useful for us too.

\begin{theorem}[Ellers and Gordeev {\cite{Ellers-Gordeev1998}}]\label{fact:Ellers}Let $H$ be a finite simple group of (twisted) Lie type over a field with more than $8$ elements. Then there exists a conjugacy class $C$ such that $H=CC$. \end{theorem}

\subsubsection{Semi-simple pseudofinite groups}

Let $H$ be a \emph{finite} group and recall that we denote by ${\rm Rad}(H)$ the \emph{solvable radical} (i.e. the maximal solvable normal subgroup) of $H$. In \cite{Wilson2009}, Wilson proved that there exists a parameter-free first-order formula $\phi_{R}(x)$ so that $\phi_{R}(H)= \{x : \phi(x)\}={\rm Rad}(H)$. 

Let $G\equiv \prod_{i\in I} G_i /U$ be pseudofinite and define $R(G)=\{g \in G : \phi_{R}(g)\}$. The group $R(G)\equiv \prod_{i\in I} {\rm Rad}(G_i) /U$ is a definable normal subgroup of $G$ which contains all definable solvable normal subgroups of $G$. In general, the subgroup $R(G)$ needs not to be solvable as there may not be a common bound on the derived lengths of the finite solvable groups ${\rm Rad}(G_i)$. The subgroup $R(G)$ is trivial if and only if $G$ is \emph{semi-simple}, that is, has no non-trivial abelian normal subgroup (see \cite[Lemma 2.17]{Point-Houcine2013}). So the notion of $R(G)$ allows one to split any pseudofinite group $G$ to a semi-simple part $G/R(G)$ and to a `solvable-like' part $R(G)$. Note that, unlike simplicity, semi-simplicity is a first-order expressible property by the first-order sentence $(\forall x\neq 1)(\exists y)([x, x^y]\neq 1)$.

\subsubsection{Finitely generated pseudofinite groups} The following natural question was raised by Sabbagh.

\begin{question}[Sabbagh]\label{question:sabbagh}Are there any finitely generated pseudofinite groups?
\end{question}

There are several non-existence results related to Question~\ref{question:sabbagh}, see \cite[Proposition 3.14]{Point-Houcine2013}. In particular, the following in well-known.

\begin{fact}[{See e.g. \cite[Proposition 3.14]{Point-Houcine2013}}]\label{fact:Point-and-Houcine-fiitely-gen} Neither solvable nor simple finitely generated pseudofinite groups exist.
\end{fact}

The following fact will be useful to us later on.

\begin{fact}[Palac\'{i}n {\cite[Lemma 2.2]{Palacin2018}}]\label{fact:Palacin} Let $G$ be a finitely generated pseudofinite group and assume that it contains an infinite definable definably simple non-abelian normal subgroup $N$. Then the group $N$ is a finitely generated pseudofinite group.

\end{fact}

\subsection{(Pseudo)finite groups of finite centraliser dimension}Recall that a group $G$ is of centraliser dimension $k$ if the longest proper chain of descending (equiv. ascending) centralisers has length $k$ and that this is denoted by ${\rm cd}(G)=k$.

Clearly, one may express in a first-order way that a group is of centraliser dimension (at most) $k$, for some \emph{fixed} $k$. So, by \L os's theorem, a pseudofinite group $G\equiv \prod_{i\in I} G_i/U$ is of centraliser dimension (at most) $k$ if and only if the finite group $G_i$ is of centraliser dimension (at most) $k$, for almost all $i$.

The following result by Khukhro immediately implies that the subgroup $R(G)$ of a pseudofinite group $G$ with ${\rm cd}(G)=k$ is solvable of $k$-bounded derived length. 

\begin{theorem}[Khukhro {\cite[Theorem 2]{Khukhro2009}}]\label{fact:khukhro} Let $G$ be a group elementarily equivalent to an ultraproduct of finite solvable groups. If $G$ is of ${\rm cd}(G)=k$ then $G$ is solvable of $k$-bounded derived length.
\end{theorem}

Another important result for us is the following.

\begin{fact}[Buturlakin {\cite[Proposition 2.1]{Buturlakin2019}}, see also \cite{Borovik-Karhumaki2019}]\label{fact:buturlakin-socle} Let $H$ be a finite group with ${\rm cd}(H)=k$. Then the centraliser dimension of $H/{\rm Rad}(H)$ is $k$-bounded. 
\end{fact}

It follows immediately from Fact~\ref{fact:buturlakin-socle} that if $G$ is a pseudofinite group with ${\rm cd}(G)=k$, then the definable quotient $G/R(G)$ is of $k$-bounded centraliser dimension.

\begin{remark}When we say that a group $G$ is of \emph{$k$-bounded} centraliser dimension (resp. has $k$-bounded derived length) we mean that the finite centraliser dimension of $G$ depends on $k$ only (resp. $G$ is solvable of derived length depending on $k$ only).
\end{remark}

\section{Proofs of Theorems~\ref{th:pseudofinite_fcd} and \ref{th:finite-gene-fcd}}\label{Sec:Proofs}

In this section we prove Theorems~\ref{th:pseudofinite_fcd} and \ref{th:finite-gene-fcd}.

Recall that the socle ${\rm Soc}(H)$ of a group $H$ is the subgroup generated by all minimal normal non-trivial subgroups. We say that ${\rm Soc}(H)=1$, if $H$ has no minimal normal non-trivial subgroups. If $H\neq 1$ is a finite group, then its socle ${\rm Soc}(H) \neq 1$ is a finite direct product of simple finite groups. 

The following easy fact is rather useful; we write down its proof for the readers convenience.

\begin{fact}\label{lemma:socle}Let $G\equiv \prod_{i\in I} G_i/U$ be a pseudofinite group. Assume that there is $d\in \mathbb{N}$ so that $| {\rm Soc}(G_i/{\rm Rad}(G_i))|\leqslant d$ for almost all $i$. Then $G/R(G)$ is finite.
\end{fact}
\begin{proof}Denote $\overline{G}=G/R(G) $ and $\overline{G}_i=G_i/{\rm Rad}(G_i)$. 

The following holds for almost all $i$. We may assume that $\overline{G}_i \neq 1$ and hence $1 \neq |{\rm Soc}(\overline{G}_i)| \leqslant d$. This means that there is a common bound on the sizes of the automorphism groups $|{\rm Aut}({\rm Soc}(\overline{G}_i))|$. As the socle ${\rm Soc}(\overline{G}_i)$ is a characteristic subgroup of $\overline{G}_i$, the finite group $\overline{G}_i$ embeds into ${\rm Aut}({\rm Soc}(\overline{G}_i))$. So there is a common bound on the sizes $|\overline{G_i}|$ and hence $\overline{G} \equiv \prod_{i\in I} \overline{G}_i /U$ is finite.  \end{proof}

\subsection{Proof of Theorem~\ref{th:pseudofinite_fcd}} Let $G$ be a pseudofinite group of ${\rm cd}(G)=k$. Denote $\overline{G}=G/R(G)$ and $\overline{G}_i=G_i/{\rm Rad}(G_i)$.

Observe first that $(\ref{radical})$ holds by Theorem~\ref{fact:khukhro}.

As the finite groups $\overline{G}_i$ are semi-simple, the socles $\overline{S}_i={\rm Soc}(\overline{G}_i)$ are finite direct products $\overline{S}_{1_i}\times \cdots \times \overline{S}_{r_i}$ of non-abelian simple groups $\overline{S}_{j_{i}}$. We start by observing that for almost all $i$ the number $r_i$ of the components $\overline{S}_{j_{i}}$ is bounded by some $\ell$ depending on $k$ only.

\begin{claim} \label{r_i} For almost all $i$, there is $\ell \in \mathbb{N}$, depending on $k$ only, such that $r_i \leqslant \ell$.
\end{claim}

Assume contrary. Then, for almost all $i$, $r_i > \ell$ for all $\ell\in \mathbb{N}$. So, by setting $j=1,\ldots, \ell$, one may choose an element $x_{j_i} \in \overline{S}_{j_i} \setminus \{1\}$ so that \[C_{\overline{S}_i}(x_{1_i})  > C_{\overline{S}_i}(x_{1_i},x_{2_i}) > \ldots > C_{\overline{S}_i}(x_{1_i},x_{2_i}, \ldots, x_{\ell_i})\] is a proper descending chain of centralisers in $\overline{G}_i$, of length $\ell$, for an arbitrary large $\ell$. But the finite group $G_i$ is of centraliser dimension $k$, and therefore, by Fact~\ref{fact:buturlakin-socle}, the quotient $\overline{G}_i$ is of centraliser dimension $k_i \in \mathbb{N}$ for some $k_i$ depending on $k$. This contradiction proves that there is $\ell\in \mathbb{N}$ so that $r_i \leqslant \ell$ for almost all $i$. Clearly, the bound $\ell$ depends on $k$ only. \qed

We may now assume that for almost all $i$, the socle $\overline{S}_i$ is a direct product of $\ell$ many simple non-abelian finite groups. Next we observe that $\overline{L}\equiv \prod_{i\in I} \overline{S}_i/U$ is a direct product of $\ell$ many simple non-abelian groups. 

\begin{claim} \label{product} $\overline{L}\equiv \prod_{i\in I} \overline{S}_i/ U$ is a direct product of $\ell$ many simple non-abelian groups $\overline{L}_{j}$, each of which is either finite or isomorphic to a (twisted) Chevalley group over a pseudofinite field. In particular, $\overline{L}$ is a definable normal subgroup of $\overline{G}$.
\end{claim}

At this point we know that, for almost all $i$, the finite group $\overline{S}_i=\overline{S}_{1_i}\times \cdots \times \overline{S}_{\ell_i}$ is a product of $\ell$ many simple non-abelian groups. Let $j \in \{1,\ldots, \ell \}$ and consider $\overline{L}_j\equiv\prod_{i\in I} \overline{S}_{j_i} /U$. If $\overline{L}_j$ is finite then, for almost all $i$, $\overline{S}_{j_i} $ is the same non-abelian finite simple group, and hence, $\overline{L}_j \cong \overline{S}_{j_i}$ is a finite non-abelian simple group. 

Assume that $\overline{L}_j$ is infinite. We may assume that just one of the three families of non-abelian finite simple groups occurs in the ultraproduct $\overline{L}_j$. This family cannot be the sporadic groups as there are only finitely many of them. One may also easily observe that this family cannot be the alternating groups; if for almost all $i$ the finite group $\overline{S}_{j_i}$ was isomorphic to some alternating group ${\rm Alt}_{n_i}$ for $n_i \geqslant 5$, then the sizes of ${\rm Alt}_{n_i}$'s were of increasing orders as $\overline{L}_j$ is infinite. However, as the centralisers of products of even number of disjoint transpositions in ${\rm Alt}_{n_i}$ form a proper centraliser chain of length at least $\lfloor \frac{n_i}{4} \rfloor$, the ultraproduct $\overline{L}_j$ of alternating groups of increasing orders would not be of finite centraliser dimension. So $\overline{S}_{j_i}$ is a simple group of Lie type over a finite field, for almost all $i$. 

As there are only finitely many different Lie types, we may assume that for almost all $i$ the groups $\overline{S}_{j_i}$ are of the same Lie type $X$. We still need to ensure that the Lie ranks $n_i$ of $\overline{S}_{j_i}$'s cannot be arbitrarily large. This follows from the structures of finite simple groups of Lie type; as explained in \cite[page 554, proof of Claim 2]{Tent-Macpherson2007}, if there is no bound on the Lie ranks $n_i$, then $\overline{L}_j$ is not of finite centraliser dimension. Therefore, for almost all $i$, the groups $\overline{S}_{j_i}$ are of the same Lie type $X$ with fixed Lie rank, say $n$, over finite fields of increasing orders. So, by Theorem~\ref{fact:Wilson-and-Ryten}, $\overline{L}_j$ is simple and isomorphic to a (twisted) Chevalley group over a pseudofinite field.  

Finally, we need to ensure that each component $\overline{L}_j\equiv \prod_{i\in I} \overline{S}_{j_i} /U$ is definable in $\overline{G}$. Again, we may assume that $\overline{L}_j$ is infinite. Then, for almost all $i$, the group $\overline{S}_{j_i}$ is a simple group of Lie type over a field of more than $8$ elements. Therefore, by Theorem~\ref{fact:Ellers}, the following holds for almost all $i$. There is $x_{j_i} \in \overline{S}_{j_i}$ such that $$\overline{S}_{j_i}=x_{j_i}^{\overline{S}_{j_i}}x_{j_i}^{\overline{S}_{j_i}}.$$ As $\overline{S}_{j_i} \unlhd \overline{G}_i$, we have $$x_{j_i}^{\overline{G}_{i}}x_{j_i}^{\overline{G}_{i}}  = \overline{S}_{j_i}.$$ This means that $\overline{S}_{j_i}$ is uniformly definable in $\overline{G}_i$, for almost all $i$. It follows that each component $\overline{L}_j$ is a definable subgroup of $\overline{G}$. We have observed that $\overline{L}$ is a definable subgroup of $\overline{G}$. Moreover, as $\overline{S}_i \unlhd \overline{G}_i$, we have $\overline{L} \unlhd \overline{G}$. \qed

At this point, setting $L$ to be the full pre-image of $\overline{L}$ in $G$, parts $(\ref{radical})$--$(\ref{socle})$ of Theorem~\ref{th:pseudofinite_fcd} are proven. We finish the proof by observing $(\ref{centraliser})$.

\begin{claim}$C_{\overline{G}}(\overline{L})=1$.
\end{claim}

We first observe that $C_{\overline{G}}(\overline{L}) \cap \overline{L}=1$. Assume contrary. Then, as centralisers are definable in a group of finite centraliser dimension, the centraliser $C_{\overline{L}}(\overline{L})= \overline{C}$ is a non-trivial definable abelian normal subgroup of $\overline{L}$. Thus $\overline{C} \equiv \prod_{i \in I} \overline{C}_i / U$, where $\overline{C}_i$ is an abelian normal subgroup of the socle $\overline{S}_{i}$. But $\overline{S}_{i}$ is a direct product of $\ell$ many simple non-abelian groups $\overline{L}_{j_i}$. That is, $\overline{C}_i \cap \overline{L}_{j_i}=1$ for all $j \in \{1,\ldots, \ell \}$; this contradiction proves that $\overline{C}=1$. 

Clearly $C_{\overline{G}}(\overline{L})$ is a normal definable subgroup of $\overline{G}$. Therefore, if $C_{\overline{G}}(\overline{L})  \neq 1$, then $C_{\overline{G}}(\overline{L}) \equiv \prod_{i\in I} \overline{K}_i / \mathcal{U}$, where $\overline{K}_i$ is a normal subgroup of $\overline{G}_i$ intersecting $\overline{L}_i$ trivially. But this is impossible as then $C_{\overline{G}}(\overline{L})$ would intersect $\overline{L}$ non-trivially. \qed

Theorem~\ref{th:pseudofinite_fcd} is now proven. \qed

\subsection{Proof of Theorem~\ref{th:finite-gene-fcd}} Let $G\equiv \prod_{i\in I} G_i/U$ be a finitely generated pseudofinite group of finite centraliser dimension. We retain our earlier notation, that is, $\overline{G}=G/R(G)$, $\overline{G}_i=G_i/{\rm Rad}(G_i)$, $\overline{S}_i={\rm Soc}(\overline{G}_i)$, $\overline{L}\equiv \prod_{i\in I}\overline{S}_i/U$ and $L$ is the full pre-image of $\overline{L}$ in $G$.

By Theorem~\ref{th:pseudofinite_fcd}, $G$ has normal series of definable subgroups $$1\unlhd R(G) \unlhd L \unlhd G,$$ where $R(G)$ is solvable and $\overline{L}$ is a finite direct product of $\ell$ many non-abelian simple groups $\overline{L}_j$, each of which is either finite or pseudofinite. Moreover, $C_{\overline{G}}(\overline{L}) = 1$. Therefore, the group $G$, in its action on $\overline{L}$ by conjugation, permutes the simple subgroups $\overline{L}_j$ and the kernel of this permutation action, say $G^\circ$, is a normal subgroup of finite index in $G$. Without a loss of generality, we may assume that $G^\circ = G$ and each $\overline{L}_j \unlhd \overline{G}$.

\begin{claim}$\overline{L}$ is finite.
\end{claim}

Towards a contradiction, assume that $\overline{L}$ is infinite. Then, for some $j\in \{1,\ldots, \ell \}$, the non-abelian simple group $\overline{L}_j$ is infinite and hence isomorphic to a (twisted) Chevalley group over a pseudofinite field. Therefore, the group $\overline{L}_j$ is an infinite non-abelian (definably) simple normal definable subgroup of the finitely generated group $\overline{G}$. So, by Fact~\ref{fact:Palacin}, $\overline{L}_j$ is finitely generated. Now $\overline{L}_j$ is a finitely generated simple pseudofinite group---this is a contradiction in light of Theorem~\ref{fact:Point-and-Houcine-fiitely-gen}.  \qed

As $\overline{L}$ is finite there is an upper bound on the sizes $|\overline{S}_i|$, for almost all $i$. Therefore, by Fact~\ref{lemma:socle}, $\overline{G}$ is finite. This means that the subgroup $R(G)$ is a solvable finitely generated pseudofinite group, contradictory to Theorem~\ref{fact:Point-and-Houcine-fiitely-gen}. Theorem~\ref{th:finite-gene-fcd} is now proven. \qed

\addcontentsline{toc}{section}{References}    
\bibliographystyle{plain}
\bibliography{ulla2021}

\end{document}